\documentclass[12pt]{amsart}

\usepackage{}
 \usepackage{mathrsfs}
\usepackage{amsfonts, amssymb,amsmath, amsthm, amsxtra, latexsym, amscd}
\usepackage[latin1]{inputenc}

\def\draft{\centerline{(Draft {\the \day}/{\the\month} \the \year.)}}

\theoremstyle{plain}
\newtheorem*{thm*}{Theorem}

\newtheorem{thm}{Theorem}[section]
\newtheorem{cor}[thm]{Corollary}
\newtheorem{lem}[thm]{Lemma}
\newtheorem{prop}[thm]{Proposition}

\newtheorem{rem}[thm]{Remark}
\newtheorem{defi}[thm]{Definition}
\numberwithin{equation}{section}

\bigskip

\def\refn#1.#2{\expandafter\def\csname#1\endcsname{[#2]}}
\def\refnr#1.{\csname#1\endcsname}

\begin{document}

\baselineskip  1.25pc

\title[essential  normality  of   principal  submodules]
{A harmonic analysis approach to essential  normality  of   principal  submodules
}
\author{Ronald G. Douglas and Kai Wang }
\address{  Department of Mathematics,
Texas A\&M University, College Station, Texas, USA}
\email{rdouglas@math.tamu.edu }
\address{School of Mathematical Sciences,
Fudan University, Shanghai, P. R. China}
\email{kwang@fudan.edu.cn}

\begin{abstract}
   Guo and the second author have  shown that the closure $[I]$ in the Drury-Arveson space
 of a homogeneous principal  ideal $I$ in $\mathbb{C}[z_1,\cdots,z_n]$ is essentially normal.
  In this note,
  the authors
    extend
   this result to the closure of
  any principal polynomial ideal in the Bergman space. In particular,    the commutators and cross-commutators of the restrictions of
  the multiplication operators are shown to be in the Schatten $p $-class for $p>n$.
  The same is true for modules generated by  polynomials with vector-valued coefficients.  Further, the maximal ideal space $X_I$ of the resulting
  $C^\ast$-algebra for
  the quotient module is shown to be contained in $Z(I)\cap \partial\mathbb{B}_n$, where $Z(I)$  is the zero
  variety for $I$, and to contain all points in $\partial\mathbb{B}_n$ that are limit points of $Z(I)\cap \mathbb{B}_n$.
   Finally, the techniques introduced enable one
  to study a certain class
  of weight Bergman spaces on the ball.

\end{abstract}
\maketitle

\section{Introduction}
  In \cite{Ar2,Ar5} Arveson raised the interesting question of whether homogeneous polynomial ideals lead to  $C^*-$algebras of essentially normal operators.
  In particular, one knew that for Hilbert spaces of holomorphic functions on the open unit ball $  \mathbb{B}_n=\{z\in\mathbb{C}^n:|z|<1\}$  such as the
  Hardy and Bergman spaces, the operators defined to be  multiplication  by polynomials were essentially normal.  Arveson focused on a related space, now called the Drury-Arveson
  space, $H^2_n$, and showed the same was true. Moreover, he asked if the submodule $[I]$ defined as the closure of a homogeneous polynomial ideal $I$   has the same property. Actually, the commutators
   and cross-commutators of these multiplication operators on $H^2_n$ are in the Schatten $p-$class $\mathcal{L}^p$ for $p>n$ and Arveson asked if the same was true for the operators on $[I]$. Perhaps the best result responding to this question is due to Guo and the second author \cite{GW}, which established that Arveson's conjecture is valid for principal \emph{homogeneous}
polynomial ideals.
In this paper, we introduce a new approach to this problem based on covering techniques from harmonic analysis.  We use it to extend the earlier result to arbitrary principal polynomial ideals.
\begin{thm*} If $\mathcal{M}=[p]$ is the submodule of the Bergman space $L_a^2(\mathbb{B}_n)$ generated by an analytic polynomial $p$, then $\mathcal{M}$ is $p-$essentially normal for $ p>n$.
\end{thm*}
As in \cite{GW}, we show that the $p-$essential normality extends to submodules generated by a polynomial with vector-valued coefficients.

 Although   the overall strategy in this paper is similar to that used in \cite{GW}, the techniques
used in this paper are very different and, we believe, provide better insight into why the result is true. In particular, the key step in the proof in \cite{GW} is an inequality which allows one to show that the
 commutators in   question are in the Schatten $ p-$class $\mathcal{L}^p$. We refer the reader to the   discussions in \cite{Esch,Sha}. An
 attempted  proof of this inequality, using standard techniques from PDE, fails since the estimate obtained only shows that those operators are
 bounded. Hence a different approach was used in \cite{GW},  but  one which was far from transparent.

Here we take advantage of the fact that the analysis takes place not just in the context of real analytic functions but for holomorphic ones. Hence, we are able to replace the inequality by one involving both the radial and complex tangential derivatives and then modify and extend known techniques from harmonic analysis  to obtain the desired result. The key step in this proof   rests  on   weighted norm estimations, which follow from a covering argument, now standard in harmonic analysis,  due to Grellier \cite{Gr}. This approach provides a new  proof for the case of principal homogeneous polynomial ideals. However, for  general   polynomials, there is still a  critical  step needed.  To handle this case, one must replace  the quantity estimated in the basic inequality by an infinite series of terms, each one  of which requires   an estimate  involving an analogue  of an inequality that follows from this  covering argument. To show that the series converges, one needs to examine  carefully how  the constants in the estimates  behave  and show that they depend  only on the dimension of the ball and the degree of the polynomial.

As a consequence of the essential normality of the cyclic submodule generated by a polynomial, one obtains an extension of the $C^*-$algebra of   compact operators
 by the   algebra of continuous functions on a closed subset of the unit sphere in $\mathbb{C}^n$ which is related to the zero variety of the polynomial. (Here one is using the quotient module defined by $[p]$.) As a result one obtains an odd K-homology element.
 We discuss these issues as well as  other consequences   of the main result. In particular, the main result is equivalent to the fact  that for the   Bergman space defined relating to the volume  measure weighted by the square of the absolute value of the polynomial, the commutators
 of the multiplication operators by coordinate  functions on this  closure are in $\mathcal{L}^p$ for $p>n$.   The result involves an explicit characterization of the elements in the spaces.

In Section 2 we provide the variant inequality, state the norm estimates  required and outline the argument    of the main result. The norm estimates are   established by an appropriate covering argument  in Section 3. Finally, in Section 4 we discuss briefly the result for the weighted Bergman space and some of the consequences of essential  normality including the odd K-homology element defined.

\section{Main Result}

 In this paper, we are mainly concerned with   the (weighted)
 Bergman spaces $L_a^2(\mathbb{B}_n)$ ($L_{a,t}^2(\mathbb{B}_n)$)  over the unit ball $\mathbb{B}_n$.
The weighted Bergman space $L_{a,t}^2(\mathbb{B}_n)$ ($t\geq 0$) consists of the analytic functions in $L_{t}^2(\mathbb{B}_n)$ with  the norm
$$\|f\|_t^2=\int_{\mathbb{B}_n} |f(z)|^2 c_t (1-|z|^2)^t dv(z),$$
where $c_t=\frac{(n+t)!}{n!t!}, dv(z)=\frac{dm(z)}{Vol(\mathbb{B}_n)}$ and $dm$ is the Lebesgue measure over $\mathbb{B}_n$, $Vol(\cdot)$ is the   measure   of the domain.  (In this paper we need only   the case that  $t$ is a non-negative integer.)
It's well known  that $L_{a,t}^2(\mathbb{B}_n)$ has the canonical orthogonal basis $\{z^\alpha:\alpha=(\alpha_1,\cdots,\alpha_n)\in \mathbb{Z}^n,\alpha_i\geq 0\,\,for\,\,1\leq i\leq n\}$
(see  e.g.  \cite{Zhu}) with $$\|z^\alpha\|_t^2=\frac{\alpha ! (n+t)!}{(n+t+|\alpha|)!}, $$
where $\alpha !=\alpha_1!\cdots\alpha_n!$ and $|\alpha|=\alpha_1+\cdots+\alpha_n$ for a multi-index $\alpha=(\alpha_1,\cdots,\alpha_n).$

  We will focus on the operators on $L_{a,t}^2(\mathbb{B}_n) $ rather than  the function theory. We pursue the same basic strategy  as in \cite{GW}. For $f\in H^\infty (\mathbb{B}_n)$, the set of all bounded analytic functions on $\mathbb{B}_n$, define  the multiplication operator $M^{(t)}_f$ on $L_{a,t}^2(\mathbb{B}_n)$ as
 $$  M^{(t)}_f(g)=fg,g\in L_{a,t}^2,$$
 which is a bounded operator with   norm $\|f\|_{\infty}.$ And  define the  weighted Toeplitz operator $T^{(t)}_f$ on $L_{a,t}^2(\mathbb{B}_n)$ with the  symbol $f\in L^\infty(\mathbb{B}_n)$ as
         $$T^{(t)}_f(g)=P^{(t)}M^{(t)}_f(g)= P^{(t)}(fg), g\in  L_{a,t}^2,$$
         where $P^{(t)}$ is the orthogonal projection from  $L_{t}^2(\mathbb{B}_n)$ to $L_{a,t}^2(\mathbb{B}_n).$  To simplify the notation, we  let $\|f\|,M_f,T_f$  denote the norm of $f$, the multiplication operator and the Toeplitz operator           on $L_a^2(\mathbb{B}_n)$, respectively.

 In this section we will prove that the cyclic submodule $\mathcal{M}=[p]$, which is  generated by an analytic polynomial $p$ in the Bergman space $L_a^2(\mathbb{B}_n)$, is essentially normal ($p-$essentially normal ). That is, the commutators $[S_{z_i},S_{z_j}^\ast]$ are compact (in $\mathcal{L}^p$) for $1\leq i,j\leq n$, where $S_{z_i}$ is the restriction of $M_{z_i}$ to $\mathcal{M}$.

  In what follows denote by $N$  the number operator on $L^2_a(\mathbb{B}_n)$ as in \cite{Ar,GW} so that $N(z^\alpha)=|\alpha|z^\alpha$ for any non-negative multi-index $\alpha$, and let $ \partial_i=\partial_{z_i},\bar{\partial}_i=\partial_{\bar{z_i}}$  be the partial derivatives  with respect to $z_i,\overline{z_i}$, respectively. Furthermore, let $R(f)=\sum_{i=1}^n z_i \partial_i(f)$ be the radial derivative. Obviously,  $Rf=mf$ for any  homogeneous analytic  polynomial $f$ with   $m=\deg(f)$. We refer the reader to \cite{Zhu} for more properties of the radial derivative. Finally, let $L_{j,i} p=  \bar{z_i} \partial_j p- \bar{z_j} \partial_i p$ be the complex tangential derivative, which behaves well relative  to the distance to the boundary as shown, for example, in \cite{Gr}, \cite[Section 7.6]{Zhu} as well as in  other  references.

  Our first result is    a variant of   formula $(2.6)$ in \cite{GW}, which is an identity relating the commutator of multiplication operators and the radial derivative.
\begin{prop}
For   analytic polynomials  $f,p\in\mathbb{C}[z_1,\cdots,z_n]$,   the equation
$$M^*_{z_j} M_p\,f-M_p M^*_{z_j}\,f=\sum_{k=0}^\infty \frac{1}{(N+1+n)^{k+1}}  [(M_{\partial_j R^{k} p}- M^*_{z_j} M_{R^{k+1} p})\,f],\, 1\leq j\leq n$$
 holds  on the Bergman space $L_a^2(\mathbb{B}_n)$.
\end{prop}
\begin{proof}
By   linearity, it is enough to verify  the case in which $p=z^\alpha$ and $\,f=z^\beta$. Using the fact that $M^*_{z_j}(z^\alpha)=\frac{\alpha_j}{n+|\alpha|}z^{\alpha-\varepsilon_j},$
  where $1\leq j\leq n$ and $\varepsilon_j$ is the multi-index with a  $1$ in the $j$ position and  $0$ in all other positions, then we have
 \begin{eqnarray*}
LHS &=&M^*_{z_j}z^{\alpha+\beta}-z^\alpha M^*_{z_j}z^{ \beta}
 = [\,\frac{\alpha_j+\beta_j}{n+|\alpha|+|\beta|}-\frac{ \beta_j}{n+ |\beta|}\,]z^{\alpha+\beta-\varepsilon_j} \\
&=& \frac{\alpha_j(n+|\beta|)-\beta_j|\alpha|}{(n+|\alpha|+|\beta| )(n+ |\beta|)}z^{\alpha+\beta-\varepsilon_j}.
 \end{eqnarray*}
Furthermore, we have
  \begin{eqnarray*}
RHS  &=&  \sum_{k=0}^\infty \frac{1}{(N+1+n)^{k+1}} \big{[}|\alpha|^k M_{\partial_j (z^\alpha)} z^\beta-M^*_{z_j}(|\alpha|^{k+1}z^{\alpha+\beta})\big{]}\\
     &=& \sum_{k=0}^\infty \frac{|\alpha|^k}{(|\alpha|+|\beta|+n)^{k+1} } [\alpha_j - \frac{|\alpha|(\alpha_j+\beta_j)}{n+|\alpha|+|\beta|} ]z^{\alpha+\beta-\varepsilon_j}\\
     &=& \frac{1}{n+|\beta|} [\alpha_j - \frac{|\alpha|(\alpha_j+\beta_j)}{n+|\alpha|+|\beta|} ]z^{\alpha+\beta-\varepsilon_j}=LHS,
 \end{eqnarray*}
which completes the proof.
\end{proof}
To use the  strategy of \cite{GW}, we need to show the convergence  of the infinite sum in  the $RHS$  above in an appropriate sense. The following proposition will play an important role in that.
\begin{prop}\label{eq:original}
For  positive integers $n$ and $m$, there is a positive constant $\small{C(n,m)>1}$ such that  for every analytic polynomial $p\in\mathbb{C}[z_1,\cdots,z_n]$ with   degree $m,$  the following inequalities hold:
\begin{eqnarray*}
&(1)&\|(R^l p)\,f\|_{2k}^2
\leq \frac{c_{2k} {C(n,m)}^{k+1}}{c_{2k-2l}} \|pf\|_{2k-2l}^2, \text{ for  every integer  } l  \text{ with }  0\leq l\leq k,        \\
&(2)& \|(L_{j,i}p) \,f \|^2_{2k+1}
\leq \frac{c_{2k+1} {C(n,m)}^{k+1}}{c_{2k}} \|pf\|_{2k}^2,  \text{ for integers } i,j  \text { with } 1\leq i\neq j\leq n,     \\
&(3)& \|  (\partial_j p)   \,f \|_{2k+2}^2
\leq \frac{c_{2k+2} {C(n,m)}^{k+1}}{c_{2k}} \|pf\|_{2k}^2,  \text{ for every integer  }   j  \text { with } 1\leq  j\leq n,
\end{eqnarray*}
   for any analytic polynomial $f\in\mathbb{C}[z_1,\cdots,z_n]$ and non-negative integer $k$, where $c_t=\frac{(n+t)!}{n!t!}$ for $t\in\mathbb{N}$.
\end{prop}

 The proof of this proposition rests heavily  on techniques from harmonic   analysis. We postpone  the proof to the next section. We show first how to obtain the essential normality of $\mathcal{M}=[p]$ from it.

\begin{lem}
Fix  $l\in\mathbb{N}$. For  any analytic polynomial $f$ satisfying  $\partial_\alpha f(0)=0$ for $|\alpha|<l$ and any non-negative integer $k$, we have
\begin{eqnarray*}
 &(1)& \|\frac{1}{(N+1+n)^{k+1/2}} f\| ^2\leq \frac{(n+2k+1+l)^{l }}{(l+1+n)^{2k+1 }}  \|f\|_{ 2k+1}^2,\,\,\,\text{and}\\
 &(2)& \|\frac{1}{(N+1+n)^{k+1/2}} [T_{z_j}^*  - T^{(2k+1) *}_{z_j}] (f)\| ^2\leq \frac{(n+2k+2+l)^{l}}{(l +n)^{2k+1 }} \|f\|_{ 2k+2}^2, 1\leq j\leq n.
  \end{eqnarray*}
\end{lem}
\begin{proof}
By  the orthogonality of homogeneous polynomials of different degrees, it's enough to show the inequality in the case that $f$ is an analytic homogeneous polynomial with $d=\deg(f)\geq l.$

(1) By the fact that  $\|z^\alpha\|_{t}^2=\frac{\alpha!(n+t)!}{(n+|\alpha|+t)!}$, we have for a  homogeneous  analytic polynomial $f=\sum_{|\alpha|=d} a_\alpha z^\alpha$ and a non-negative integer $t$, that
\begin{eqnarray*}
\|f\|^2=\sum_{|\alpha|=d} |a_\alpha|^2 \frac{\alpha!n!}{(n+d )!}&=&\frac{ n!}{(n+d)!}\,\frac{ (n+t+d)!}{(n+t)!}\sum_{|\alpha|=d} |a_\alpha|^2 \frac{\alpha!(n+t)!}{(n+d +t)!}\\&=&\frac{ n!}{(n+d)!}\,\frac{ (n+t+d)!}{(n+t)!}\|f\|_{ t}^2.
\end{eqnarray*}
Therefore, for a non-negative integer $k$ we have
\begin{eqnarray*}
LHS^{(1)}&=&\|\frac{1}{(d+1+n)^{k+1/2}} f\|^2=\frac{1}{(d+1+n)^{2k+1 }}\|f\|^2 \\
   &=& \frac{1}{(d+1+n)^{2k+1 }} \,\frac{ n!}{(n+d)!}\,\frac{ (n+2k+1+d)!}{(n+2k+1)!}\,\|f\|_{ 2k+1}^2.\\
   \end{eqnarray*}
   Here, $LHS^{(1)}$ refers to the left-hand side of the inequality in statement $(1)$.

   Since $d\geq l$ and
   $$\frac{(n+2k+1+d)!}{(d+1+n)^{2k+1 }(n+d)! }=\frac{(n+2k+1+d)\cdots(n+d+1)}{(d+1+n)^{2k+1 }}=(1+\frac{2k}{d+1+n})\cdots (1),$$
  we see that this product  is monotonically    decreasing with respect to $d$. Thus we have
    $$ \frac{(n+2k+1+d)!}{(d+1+n)^{2k+1 }(n+d)! } \leq \frac{(n+2k+1+l)!}{(l+1+n)^{2k+1 }(n+l)! }.$$
    This means that
    \begin{eqnarray*}LHS^{(1)}&\leq& \frac{n! (n+2k+1+l)!}{(l+1+n)^{2k+1 }(n+l)!(n+2k+1)!}   \|f\|_{ 2k+1}^2 \\ &\leq& \frac{(n+2k+1+l)^{l }}{(l+1+n)^{2k+1 }}  \|f\|_{ 2k+1}^2=RHS^{(1)},\end{eqnarray*}
    which completes the proof of $(1)$.

(2) We begin  the proof of $(2)$ with an observation.  Although  the range of $T^{(2k+1) *}_{z_j}$ is contained in $L_{a,2k+1}^2(\mathbb{B}_n)$, it's easy to see that  the image of an analytic  polynomial   under $T^{(2k+1) *}_{z_j}$ is still  an analytic  polynomial. This follows from   the fact that
\begin{equation}\label{eq:temp}
T^{(2k+1) *}_{z_j}(z^\alpha)=\frac{\alpha_j}{n+2k+1+|\alpha|}z^{\alpha-\varepsilon_j} ,\,1\leq j\leq n.
\end{equation}
Therefore it belongs to  $L_{a }^2(\mathbb{B}_n) $ and the LHS$^{(2)}$   makes sense if $f$ is an analytic polynomial. Specializing (\ref{eq:temp}) to $k=0$, one sees that
 \begin{eqnarray}T_{z_j}^* (z^\alpha)=\frac{\alpha_j}{n+|\alpha|}z^{\alpha-\varepsilon_j} ,\,1\leq j\leq n.\end{eqnarray}
Combining   formulas   $(\ref{eq:temp}),(2.2)$,  we have
 $$T_{z_j}^* (z^\alpha)-T^{(2k+1) *}_{z_j}(z^\alpha)=\frac{\alpha_j(2k+1)}{(n+|\alpha|)(n+2k+1+|\alpha|)}z^{\alpha-\varepsilon_j}=\frac{2k+1}{n+2k+1+|\alpha|}T_{z_j}^* (z^\alpha) .$$
Thus, for any  homogeneous analytic  polynomial $f$ with $d=\deg( f),$ one has that
 $$T_{z_i}^* (f)-T^{(2k+1) *}_{z_i}(f)=\frac{2k+1}{n+2k+1+d}T_{z_i}^* (f) . $$
This implies that
 \begin{eqnarray*}
LHS^{(2)}&=&\|\frac{1}{(d +n)^{k+1/2}} \frac{2k+1}{n+2k+1+d}T_{z_j}^* (f)\|^2\\
    &\leq& \frac{1}{(d +n)^{2k+1 }}\frac{(2k+1)^2}{(n+2k+1+d)^2}\|f\|^2 \\
   &=& \frac{1}{(d +n)^{2k+1 }}\frac{(2k+1)^2}{(n+2k+1+d)^2} \,\frac{ n!}{(n+d)!}\,\frac{ (n+2k+2+d)!}{(n+2k+2)!}\,\|f\|_{ 2k+2}^2.
  \end{eqnarray*}
  Using  the same monotonicity  argument as in $(1)$, one shows that
  $$\frac{(n+2k+2+d)!}{(n+2k+1+d)(d +n)^{2k+1 }(n+d)!} \leq  \frac{(n+2k+2+l)!}{(n+2k+1+l)(l +n)^{2k+1 }(n+l)!}.$$
Hence,
\begin{eqnarray*}
LHS^{(2)} &\leq &\frac{(2k+1)^2 n! (n+2k+2+l)!}{(n+2k+1+l)^2 (l +n)^{2k+1 }(n+l)!(n+2k+2)!}\|f\|_{ 2k+2}^2
\\ &\leq& \frac{(n+2k+2+l)^{l}}{(l +n)^{2k+1 }} \|f\|_{ 2k+2}^2=RHS^{(2)},
\end{eqnarray*}
which  completes the proof of the lemma.
\end{proof}

Using Proposition 2.2 and Lemma 2.3,  we establish  in the following proposition the necessary norm estimates for  each term appearing in the infinite sum of Proposition 2.1.

\begin{prop}
For non-negative integers $k,l$ and   analytic polynomials $p,f\in\mathbb{C}[z_1,\cdots,z_n]$ satisfying  $\partial_\alpha f(0)=0$ for $|\alpha|<l$ and $m=\deg(p)$, we have the   inequality
$${\textstyle\|\frac{1}{(N+1+n)^{k+1/2}}  [M_{\partial_j R^{k} p}- M^*_{z_j} M_{R^{k+1} p}](f)\|  \leq \frac{(n+1)(n+2k+2+l)^{ (l+n) /2} C(n,m)^{k+1 }  }{ (l+n)^{k+1/2 } } \|pf\|},    $$
  where $C(n,m)$ is the constant appearing in Proposition 2.2 which depends only on $n,m$.
\end{prop}

\begin{proof} The key idea of the proof is the following well-known identity  (see e.g. \cite{CFO} )
\begin{eqnarray}\label{eq:derivative}\partial_j g- \overline{z_j} R g&=&(1-\sum_{i=1}^n |z_i|^2) \partial_j g+ \sum_{i=1}^n z_i [\overline{z_i} \partial_j ( g)-\overline{z_j} \partial_i ( g) ]\\
\nonumber &=&(1-|z|^2) \partial_j g+ \sum_{i=1,i\neq j}^n z_i L_{j,i} ( g) \end{eqnarray}
for any smooth function $g$ on $\mathbb{B}_n$.

Using the above identity with  $g=R^k p, $ we have

 \begin{eqnarray*}
 &&\|\frac{1}{(N+1+n)^{k+1/2}}  [M_{\partial_j R^{k} p}- M^*_{z_j} M_{R^{k+1} p}](f)\|  \\
 &\leq & {\textstyle\|\frac{1}{(N+1+n)^{k+1/2}}  [M_{\partial_j R^{k} p}- T^{(2k+1) *}_{z_j}  M_{R^{k+1} p}](f)\|+
  \|\frac{1}{(N+1+n)^{k+1/2}}   ( T_{z_j}^*  -T^{(2k+1) *}_{z_j} ) M_{R^{k+1} p} (f)\|} \\
 &\leq & {\textstyle\frac{(n+l+2k+1)^{ l /2}}{ (l+1+n)^{ k+1/2 } }  \|  T^{(2k+1)  }_{\partial_j R^{k} p-\overline{z_j} R^{k+1}  p }   f\|_{2k+1}
  +\frac{ (n+2k+2+l)^{l/2} }{ (l+n)^{k+1/2 } } \|  M_{R^{k+1}  p } (f)\|_{2k+2}} \\
 &\leq&{\textstyle \frac{ (n+2k+2+l)^{ l /2} }{ (l+n)^{k+1/2 } }  }
 [\|(1-|z|^2) \partial_j R^k (p)\, f\|_{2k+1}+ \sum_{i=1,i\neq j}^n  \| L_{j,i}R^k (p)   f  \|_{2k+1}+\|   {R^{k+1} (p)}  f \|_{2k+2}]
               \\ &\leq&{\textstyle  \frac{ (n+2k+2+l)^{ l /2} }{ (l+n)^{k+1/2 } }  [n{ \sqrt{\frac{c_{2k+1}C(n,m)^{k+1}}{c_{2k}}}}\|  R^k( p ) f\|_{2k } +\|   {R^{k+1} (p)} f \|_{2k+2}]}
 \\ &\leq& \frac{(n+1)(n+2k+2+l)^{ l /2} C(n,m)^{k +2}\sqrt{c_{2k+2}} }{ (l+n)^{k+1/2 } } \|pf\|
 \\ &\leq& \frac{(n+1)(n+2k+2+l)^{ (l+n) /2} C(n,m)^{k+2 }  }{ (l+n)^{k+1/2 } } \|pf\|.
 \end{eqnarray*}
The first inequality follows from the triangle inequality, while the second one is implied by
Lemma 2.3, and that to the fourth line   follows from formula (\ref{eq:derivative}) and the triangle inequality. Finally the inequalities of the second and third lines from the end  follow  from Proposition 2.2. This completes the proof of the proposition.\end{proof}
\vskip3mm
We  now prove the essential  normality of $\mathcal{M}=[p]$ for $p\in\mathbb{C}[z_1,\cdots,z_n]$.

\begin{thm}
If $\mathcal{M}=[p]$ is the cyclic submodule of the Bergman space $L_a^2(\mathbb{B}_n)$ generated by an analytic polynomial $p\in\mathbb{C}[z_1,\cdots,z_n]$, then $\mathcal{M}$ is $p-$essentially normal for $p>n$.
\end{thm}\begin{proof}
Suppose that $m=\deg(p)$ and fix $l$ satisfying $n+l\geq 2C(n,m)$. Let $$\mathscr{E}_l=\{f\in\mathbb{C}[z_1,z_2,\cdots,z_n]:\partial_\alpha f(0)=0 \,\,for\,\,|\alpha|<l \}.$$
For any integer $j$ with $1\leq j\leq n$, define  $ D_j: p   {\mathscr{E}}_l \subset L_a^2(\mathbb{B}_n)\to L_a^2(\mathbb{B}_n)$ by
$$D_j (pf)= \sum_{k=0}^\infty \frac{1}{(N+1+n)^{k+1/2}}  [M_{\partial_j R^{k} p}- M^*_{z_j} M_{R^{k+1} p}](f),\,\,f\in{\mathscr{E}}_l.$$
By Proposition 2.4, $D_j$ is a bounded operator.

Let $P_l$ be the projection from $L^2_a(\mathbb{B}_n)$ to the closure $\mathcal{M}_l$ of  $ p \mathscr{E}_l$ in $\mathcal{M}=[p]$.  Using Proposition 2.1, we have  that for any polynomial $f\in \mathscr{E}_l$
\begin{eqnarray*}P_{\mathcal{M}^\perp} M_{z_j}^\ast P_l (pf) &=&P_{\mathcal{M}^\perp} M_p M_{z_j}^\ast(f)+P_{\mathcal{M}^\perp} \frac{1}{(N+1+n)^{1/2}}D_j(pf)\\&=&P_{\mathcal{M}^\perp} \frac{1}{(N+1+n)^{1/2}} D_j(pf).\end{eqnarray*}
This means that  $P_{\mathcal{M}^\perp} M_{z_j}^\ast P_l $ is in the Schatten $p-$class for $p>2n$ by the fact that $\frac{1}{ N+1+n  }$ is in the Schatten $ p-$class for $p>n$ as shown in  \cite{Ar} and $D_j$ is bounded.

Since $\mathcal{M}_l$ is a finite codimensional subspace of $\mathcal{M}$, for any integer $j$ with $1\leq j\leq n$ we have $P_{\mathcal{M}^\perp} M_{z_j}^\ast P_\mathcal{M}$ is also in the Schatten $ p-$class for $p>2n$. By Lemma 2.1 in \cite{GW}, one sees that  $\mathcal{M}$ is $p-$essentially normal for $p>n$.
\end{proof}

\begin{rem}
 Theorem 2.5 can be generalized to the vector-valued case with a slight modification.
  Let $\mathbf{p}=(p_1,\cdots,p_r)\in
\mathbb{C}[z_1,\cdots,z_n] \otimes\mathbb{C}^r $, where each   $p_i$ is a
polynomial with $\deg(p_i)\leq m$ for some fixed $m$, and   $\mathcal{M}=[\mathbf{p}]$ be the submodule of  $L_a^2(\mathbb{B}_n) \otimes \mathbb{C}^r$ generated by
$\mathbf{p}$.
For $1\leq j\leq n$, define $ D_j=(D_{j,1},\cdots,D_{j,r}): \mathbf{p}  {\mathscr{E}}_l \to L_a^2(\mathbb{B}_n) \otimes \mathbb{C}^r$ by
$$D_{j,i} (p_if)= \sum_{k=0}^\infty \frac{1}{(N+1+n)^{k+1/2}}  [M_{\partial_j R^{k} p_i}- M^*_{z_j} M_{R^{k+1} p_i}](f),\,\,f\in{\mathscr{E}}_l.$$
 Using an argument similar to that for Theorem 2.5, one sees that for any $f\in \mathscr{E}_l$,
$P_{\mathcal{M}^\perp} M_{z_j}^\ast P_l (\mathbf{p}f)  = P_{\mathcal{M}^\perp}\frac{1}{(N+1+n)^{1/2}} D_j(\mathbf{p}f)$.  Thus, one
can obtain that $P_{\mathcal{M}^\perp} M_{z_j}^\ast P_\mathcal{M} \in
\mathcal {L}^{ p} $ for $p>2n.$ This means that the
 submodule $[\mathbf{p}]$ is $p$-essentially normal for
$p>n.$
\end{rem}
\section{Proof of Proposition \ref{eq:original}}
We will complete the proof of Proposition 2.2 in this section by proving  an equivalent variant of it.

In what follows, we set $\Omega_r=\{z\in\mathbb{B}_n:|z|>r\}$ for $0< r< 1$.\vskip3mm

\vskip3mm{\noindent\bf{Proposition\,\,\ref{eq:original}$ {'}$.}}{\emph{
For  positive integers $n$ and $m$, there is a positive constant $\small{C(n,m)>1}$ such that  for an analytic polynomial $p\in\mathbb{C}[z_1,\cdots,z_n]$ with   degree $m,$   the following inequalities hold for any analytic polynomial $f$ and non-negative integers $i,j,k,l$ with $0\leq l\leq k,$    $1\leq i\neq j\leq n$:
 \begin{eqnarray*}
&(1)& \int_{ \Omega_{\frac{1}{2}}} |(R^l p)(z)\,f(z)|^2   (1-|z|^2)^{2k} dm(z)\leq {C(n,m)}^{k+1}  \int_{\mathbb{B}_n} |p(z)f(z)|^2  (1-|z|^2)^{2k-2l}  dm(z);\\
&(2)&\int_{\Omega_{\frac{1}{2}}} |(L_{j,i}p)(z)\,f(z)|^2   (1-|z|^2)^{2k+1} dm(z)\leq C(n,m)^{k+1} \int_{\mathbb{B}_n} |p(z)f(z)|^2   (1-|z|^2)^{2k} dm(z); \\
&(3)&\int_{\Omega_{\frac{1}{2}}}|  (\partial_j p)  (z)\,f(z)|^2   (1-|z|^2)^{  2k+2} dm(z)\leq {C(n,m)}^{k+1}  \int_{\mathbb{B}_n} |p(z)f(z)|^2   (1-|z|^2)^{2k } dm(z).
\end{eqnarray*}
}
\vskip3mm}

 Note that the constants $c_t$ appearing in the statements of Proposition 2.2 are implicit here since integrals have replaced norms in these statements. With that observation it's easy to see that {Proposition\,\,\ref{eq:original}$ {'}$} follows from {Proposition\,\,\ref{eq:original}}. We use the following lemma
 to prove the other direction.
\begin{lem}
For a non-negative integer $t$ and  $f\in L_{a,t}^2(\mathbb{B}_n)$, we have
$$ \int_{\mathbb{B}_n}  | f(z)|^2 (1-|z|^2)^{t} dm(z)   \leq  3^{t+1}  \int_{\Omega_{\frac{1}{2}}} |f(z)|^2 (1-|z|^2)^t dm(z). $$
\end{lem}

\begin{proof}
We begin with the case $t=0$.
It's easy to see that
$$ \int_{|z|<\frac{1}{2}}  |z^\alpha|^2   dm(z)  = (\frac{2}{3})^{2|\alpha|+2n}  \int_{|z|<\frac{3}{4}}  |z^\alpha|^2   dm(z). $$
Thus,
$$\int_{|z|<\frac{1}{2}}  |z^\alpha|^2   dm(z)  = \frac{(\frac{2}{3})^{2|\alpha|+2n}}{1-(\frac{2}{3})^{2|\alpha|+2n}}  \int_{\frac{1}{2}<|z|<\frac{3}{4}}   |z^\alpha|^2   dm(z)\leq 2 \int_{\frac{1}{2}<|z|<\frac{3}{4}}   |z^\alpha|^2   dm(z).$$
Therefore,  for each analytic function $f$ on $\mathbb{B}_n$, it follows that
$$ \int_{|z|<\frac{1}{2}}  |f(z)|^2   dm(z) \leq 2 \int_{\frac{1}{2}<|z|<\frac{3}{4}} |f(z)|^2    dm(z)  .$$
For the general case  $t\geq 0$, we have
\begin{eqnarray*}
& &\int_{|z|<\frac{1}{2}}  |f(z)|^2 (1-|z|^2)^{t}  dm(z) \leq \int_{|z|<\frac{1}{2}}  |f(z)|^2   dm(z)\\
 &\leq& 2 \int_{\frac{1}{2}<|z|<\frac{3}{4}} |f(z)|^2    dm(z)
\leq 3^{t+1} \int_{\frac{1}{2}<|z|<\frac{3}{4}} |f(z)|^2  (1-|z|^2)^{t}   dm(z),
\end{eqnarray*}
which  leads to the desired result.
\end{proof}

Now we show how to prove {Proposition\,\,\ref{eq:original}} from  {Proposition\,\,\ref{eq:original}$ {'}$}.

By Lemma 3.1, clearly
$(1),(3)$ in {Proposition 2.2} and {Proposition\,\,\ref{eq:original}$ {'}$} are equivalent. Inequality   $(2)$
 is not so obvious  since $L_{j,i}(p)$ is not analytic in general. To avoid unnecessary  complexity, we show that  $(2)$ and $(3)$ of {Proposition\,\,\ref{eq:original}$ { '}$}
 imply $(2)$ of {Proposition\,\,\ref{eq:original}}. In fact, $(2) $ of {Proposition\,\,\ref{eq:original}$ { '}$} implies that
 $$c_{2k+1}\int_{\Omega_{\frac{1}{2}}} |(L_{j,i}p)(z)\,f(z)|^2   (1-|z|^2)^{2k+1} dv(z)\leq \frac{c_{2k+1}  C(n,m) ^{k+1}}{c_{2k}} \|p(z)f(z)\|_{2k}^2 ;$$
 and using Lemma 3.1 and $(3) $ of {Proposition\,\,\ref{eq:original}$ { '}$} one  shows that
 \begin{eqnarray*}
&&\|(L_{j,i}p)(z)\,f(z)\|^2_{2k+1}-c_{2k+1}\int_{\Omega_{\frac{1}{2}}} |(L_{j,i}p)(z)\,f(z)|^2   (1-|z|^2)^{2k+1} dv(z) \\
&=& c_{2k+1} \int_{|z|< {\frac{1}{2}}} |(L_{j,i}p)(z)\,f(z)|^2   (1-|z|^2)^{2k+1}{\textstyle \frac{dm(z)}{Vol(\mathbb{B}_n)}} \\
 &\leq& c_{2k+1}\int_{|z|<{\frac{1}{2}}}2\big{[} |\overline{z_j}(\partial_i p)(z)\,f(z)|^2+|\overline{z_i}(\partial_j p)(z)\,f(z)|^2  \, \big{]}(1-|z|^2)^{2k+1} {\textstyle \frac{dm(z)}{Vol(\mathbb{B}_n)}}
 \\
 &\leq & 4 c_{2k+1}\int_{|z|<{\frac{1}{2}}}\big{[}\, |(\partial_i p)(z)\,f(z)|^2+|(\partial_j p)(z)\,f(z)|^2 \, \big{]} (1-|z|^2)^{2k+2} {\textstyle \frac{dm(z)}{Vol(\mathbb{B}_n)}}\\
 &\leq & 4\cdot 3^{2k+3} c_{2k+1}\int_{\Omega_{\frac{1}{2}}}[\, |(\partial_i p)(z)\,f(z)|^2+|(\partial_j p)(z)\,f(z)|^2  \, ](1-|z|^2)^{2k+2} {\textstyle \frac{dm(z)}{Vol(\mathbb{B}_n)}}\\
  &\leq & 8\cdot 3^{2k+3} c_{2k+1} C(n,m)^{k+1} \int_{\mathbb{B}_n } |  p(z)\,f(z)|^2    (1-|z|^2)^{2k} {\textstyle \frac{dm(z)}{Vol(\mathbb{B}_n)}}\\
  &\leq & \frac{c_{2k+1} (8\cdot 3^3C(n,m))^{k+1}}{c_{2k}} \|p(z)f(z)\|_{2k}^2.
 \end{eqnarray*}
Therefore, $\|(L_{j,i}p)(z)\,f(z)\|^2_{2k+1}\leq   \frac{c_{2k+1} ( 217 C(n,m))^{k+1}}{c_{2k}} \|p(z)f(z)\|_{2k}^2$, as desired, and we have shown that Propositions 2.2 and $2.2'$ are equivalent. \vskip4mm

The remainder of this section will be devoted to the  proof of the weight norm estimates in {Proposition\,\,\ref{eq:original}$ {'}$}. The  strategy of that  is   similar to the argument in \cite{Gr}. However, we will give a complete proof, since in our proof we need to keep careful track of the   constants. Let us begin   with a local result in dimension one.

\begin{lem}
 For a  one-variable analytic  polynomial   $p\in\mathbb{C}[z] $  with $m\geq \deg(p),$   an  integer $l$ with $1\leq l\leq m$ and an  analytic function $f$ on   the complex plane $\mathbb{C}$, we have

 (1) $|\partial^l p(0)\,f(0)|\leq \frac{m!}{ (m-l)!}\int_\mathbb{T} |pf| \frac{d \theta }{2\pi},$ where $\frac{d \theta }{2\pi}$ is the normalized  Lebesgue measure on the unit circle $\mathbb{T}$.

(2)
 $r^l|\partial^l p(0)\,f(0)|\leq \frac{(l+2)m!}{2 (m-l)! }\int_{r\mathbb{D}} |pf| \frac{dm(z)}{\pi r^2}$, where $\frac{dm(z)}{\pi r^2}$ is the normalized  Lebesgue measure on the disk ${r\mathbb{D}}$.
\end{lem}

\begin{proof}
(1)
Without loss of generality, suppose $m= \deg(p) $ and $$p(z)=z^u(z-a_1)\cdots(z-a_v)(z-b_1)\cdots(z-b_s),$$ where $u+v+s=m, |a_i|\geq 1,|b_i|<1,b_i\neq 0$.
It's easy to see that $|\partial^l p(0)|=0$ if $l<u$. Moreover,  for $l\geq u$ we have
\begin{eqnarray*}
|\partial^l p(0)|=| l! \sum_{\stackrel{\Lambda_1\subseteq \{1,2,\cdots, v\};}{\stackrel{\Lambda_2\subseteq \{1,2,\cdots, s\};}{|\Lambda_1|+|\Lambda_2|=m -l}} }
\prod_{i\in \Lambda_1,j\in\Lambda_2}    a_{i} b_j
| \leq  l! \sum_{\stackrel{\Lambda_1\subseteq \{1,2,\cdots, v\};}{\stackrel{\Lambda_2\subseteq \{1,2,\cdots, s\};}{|\Lambda_1|+|\Lambda_2|=m -l}} } |a_1\,\cdots a_v|\leq \frac{m!}{ (m-l)!} |a_1\,\cdots a_v|.\end{eqnarray*}
Therefore,
\begin{eqnarray*}
\int_\mathbb{T} |pf| \frac{d \theta}{2\pi} &=&\int_\mathbb{T} |(z-a_1)\cdots(z-a_v)(z-b_1)\cdots(z-b_s)f| \frac{d \theta }{2\pi}\\
&=& \int_\mathbb{T} |(z-a_1)\cdots(z-a_v)(1-\overline{b_1}z)\cdots(1-\overline{b_s}z)f| \frac{d \theta }{2\pi}\\
&\geq& |a_1\cdots a_v||f(0)|\geq \frac{  (m-l)! |\partial^l p(0) f(0)|}{m!}.
\end{eqnarray*}

(2) For $r>0$ and the analytic function $f$,  let $f_r(z)=f(rz)$. Then we have
\begin{eqnarray*}
\int_{r\mathbb{D}} |pf| \frac{dm(z)}{\pi r^2}&=& \int_{0<r'<r}\int_\theta |p(r' e^{i\theta})f(r' e^{i\theta})| \frac{r'dr'd\theta}{\pi r^2}\\
&\geq &  \int_{0<r'<r} |\frac{2\pi(m-l)! }{m!} \partial^l p_{r'}(0) f_{r'}(0)|\frac{r'dr' }{\pi r^2}\\
&=&  \frac{2(m-l)! }{m!} |\partial^l p (0) f(0)| \int_{0<r'<r} \frac{{r'}^{l+1} dr' }{  r^2}\\
&=&\frac{2 (m-l)! }{(l+2)m!} |r^l\,\partial^l p (0) f(0)|,
\end{eqnarray*}
ending the proof of the lemma.
\end{proof}

We will establish  the full inequalities in Proposition $2.2'$ using the local result from the preceding lemma and  the following Covering Lemma. We start by defining a special family of open subsets of $\mathbb{C}^n$.
\begin{defi}
For any $a\in \mathbb{C}^n -\{0\}$, let $P_a$ be the orthogonal projection from $\mathbb{C}^n$ onto the one-dimensional subspace $[a]$ generated by $a$,
 and $P_a^\perp$ be the orthogonal projection from $\mathbb{C}^n$ onto $\mathbb{C}^n\ominus [a]$. Given $ \delta>0 $, define the  neighborhood $Q_\delta(a)$ of $a$
 by
 $$Q_\delta(a)=\{z\in\mathbb{C}^n: |P_a(z)- a|<\delta, |P_a^\perp(z) |< \sqrt{\delta} \}.$$
\end{defi}\vskip2mm

\begin{lem}
Fixing $ \frac{1}{4}<r<1$ and $0<c<\min\{\frac{r-\frac{1}{4}}{4},\frac{1}{10}\}$, define $ \delta(z) = c(1-|z|)$.  For    $z\in \Omega_r$,    we have:

(1) For any $z'\in Q_{\delta(z)}(z)$,   $$1-3c< \frac{1-|z'|^2}{1-|z|^2}< 1+2c;\,\,
\frac{1}{3} < \frac{1-|z'|}{1-|z|}<  3;\,\,1-4c<\frac{|z'|}{|z|} .$$

(2) $Q_{\delta(z)}(z) \subseteq \Omega_{r-4c} \subseteq \Omega_{\frac{1}{4}}$.

(3) There exists a constant $C=200 $ independent of $z, r, c $ such that, if $  z'\in Q_{\delta(z)}(z)$, then $  Q_{\delta(z)}(z)\subset Q_{C\,\delta(z')}(z')$ and $Q_{\delta(z')}(z')\subset Q_{C\,\delta(z)}(z)$.
\end{lem}
\begin{proof}
Using rotations in $\mathbb{C}^n$, without   loss of generality we can suppose $z=(a,0,0,\cdots,0)$ and $z'=(b_1,b_2,0,\cdots,0)$ with $0<a<1,0<b_2$.

(1)By the definition of $Q_{\delta(z)}(z)$, $|b_1-a|<\delta(z)$ and $|b_2|<\sqrt{\delta(z)}$. This implies that
$$|b_1|<a+\delta(z)<a+\frac{1-a}{10}<1.  $$
Furthermore, using a direct computation one sees that
 $$\frac{1-|z'|^2}{1-|z|^2}=1+ \frac{|z|^2-|z'|^2}{1-|z|^2}=1+ \frac{a^2-|b_1|^2}{1-|z|^2}-\frac{|b_2|^2}{1-|z|^2}$$
and
$$0\leq\frac{|a^2-|b_1|^2|}{1-|z|^2} \leq \frac{(a+|b_1|)|a-b_1|}{(1+|z|)(1-|z|)}< 2c,\,\, 0\leq\frac{|b_2|^2}{1-|z|^2}< c.$$
Therefore,
  $$1-3c < \frac{1-|z'|^2}{1-|z|^2} < 1+2c.$$
  This implies that
\begin{eqnarray}\label{radius}\frac{1}{3}<  \frac{(1-3c)(1+|z|) }{1+|z'| } < \frac{1-|z'| }{1-|z| } <  \frac{(1+2c)(1+|z|) }{1+|z'| } < 3. \end{eqnarray}
Moreover, since $(1-4c) |z| <|b_1| $, we have $1-4c<\frac{|z'|}{|z|}.$

(2) From $(1)$ it follows that $1>|z'|>|z|-4c\geq r-4c\geq \frac{1}{4}$, as desired.

(3) For a point $w\in\mathbb{C}^n$, write $ w=(w_1,w_2,w')$ with $w_1,w_2\in\mathbb{C},w'\in\mathbb{C}^{n-2}$.  If $ w=(w_1,w_2,w')\in Q_{ {\delta(z')}}(z')$, then by Definition 3.3 and inequality  ($\ref{radius}$) we have that  $|w'|<\sqrt{\delta(z')}<\sqrt{3\delta(z)}$,   and
$$(w_1,w_2)=u  (b_1,b_2) +s(-b_2,\overline{b_1}) $$
with $|s|< \sqrt\frac{{\delta(z')}} {|b_1|^2+|b_2|^2} \leq 4\sqrt{\delta(z')}$  and $|(u-1)(b_1,b_2)|<\delta(z')$. This means that
 \begin{eqnarray*}|w_1-a|&=&|u b_1-s b_2-a|\leq |(u-1)b_1|+|b_1-a|+|s b_2|\\ &<&  \delta(z')+{\delta(z)}+4\sqrt{\delta(z')\delta(z)}< 16 {\delta(z)}\end{eqnarray*}
  and
   \begin{eqnarray*}|w_2|&=&|u b_2+s\overline{b_1}|\leq |(u-1)b_2|+|b_2|+ |s b_1|\\&<& \delta(z')+\sqrt{\delta(z)}+4\sqrt{\delta(z')}\leq 6 \sqrt{3\delta(z)}.\end{eqnarray*}
  So,   $Q_{\delta(z')}(z')\subset Q_{200\delta(z)}(z)$.

  On the other hand, if $ w=(w_1,w_2,w')\in Q_{ \delta(z)}(z)$,  by Definition 3.3  we have   $|w'|,|w_2|<\sqrt{\delta(z)}$ and $|w_1-a|<\delta(z) .$ A direct computation shows that
  $$(w_1,w_2)= \frac{w_1 \overline{ b_1}+w_2 b_2}{|b_1|^2+|b_2|^2}  (b_1,b_2) +\frac{w_2  {b_1}-w_1 b_2}{|b_1|^2+|b_2|^2} (-b_2,\overline{ b_1}) .$$
Since
\begin{eqnarray*}
&&|\frac{w_1 \overline{b_1}+w_2 b_2}{|b_1|^2+|b_2|^2}  (b_1,b_2) -(b_1,b_2)| \\
&\leq&  |\frac{(w_1-a) \overline{b_1} }{|b_1|^2+|b_2|^2}  (b_1,b_2)|+|\frac{ w_2 b_2}{|b_1|^2+|b_2|^2}  (b_1,b_2)|+|\frac{a \overline{b_1} }{|b_1|^2+|b_2|^2}  (b_1,b_2)-(b_1,b_2)|\\
&\leq&   5\delta(z) +4|a \overline{b_1}-|b_1|^2-|b_2|^2|\leq 13\delta(z)\leq 39\delta(z');
\end{eqnarray*}
and
$$| \frac{w_2  {b_1}-w_1 b_2}{\sqrt{|b_1|^2+|b_2|^2}}|\leq 8\sqrt{\delta(z)}\leq 8\sqrt{3\delta(z')},$$
it follows that we have $Q_{\delta(z')}(z')\subset Q_{200\delta(z)}(z)$ as desired.
\end{proof}

 \begin{prop}[Covering Lemma]\label{covering}
  Fix  $r=\frac{1}{2},\,c=\frac{1}{10\cdot 200^3} $ and define $ \delta(z) =  {c(1-|z|)} $. Then  there exists a  countable  set of points $\{z_s\} $ in $\Omega_r$ having the following properties:

  $(i)\,\,\,\Omega_r\subseteq \bigcup_s Q_{\delta(z_s)}(z_s)$ and $Q_{200^{-2} \delta(z_j )}(z_j )  \cap Q_{200^{-2} \delta(z_s )}(z_s ) =\emptyset \text{ if } j\neq s.$

  $(ii)\,\,\,Q_{200^2\,\delta(z_s)}(z_s)\subseteq \Omega_{r-c}$,  and  no point   belongs to more than $N(n)+1$ of the sets $Q_{200^2\,\delta(z_s)}(z_s), $
 where $N(n)=200^{6n+6}$ depends only on the dimension $n$.
 \end{prop}

\begin{proof}
First  we  choose $\{z_s\}$ satisfying $(i)$ by a classical method of harmonic analysis.

Set $\Gamma_1=\{Q_{200^{-2} \delta(z)}(z):z\in\Omega_r\}$. Let $r_1$ be the supremum of the radii $200^{-2} \delta(z)$ of the members $Q_{200^{-2} \delta(z)}(z)$ of $\Gamma_1$.  Choose $z_1\in\Omega_r$ with radius $200^{-2} \delta(z_1)>\frac{r_1}{2}. $ Discard  all the sets in $\Gamma_1$ that intersect $Q_{200^{-2} \delta(z_1)}(z_1),$ and denote the remaining collection by $\Gamma_2$. Let $r_2$ be the supremum of the radii  of the members   of $\Gamma_2$ and choose $z_2$ with radius $200^{-2} \delta(z_2)>\frac{r_2}{2}. $
After, discarding all the sets in $\Gamma_2$ that intersect $Q_{200^{-2} \delta(z_2)}(z_2),$ denote the remaining collection by $\Gamma_3$, and continue inductively.
One sees that the process  will continue through the natural numbers. We thus get a sequence $\{z_s\}$ such that $Q_{200^{-2} \delta(z_j )}(z_j )  \cap Q_{200^{-2} \delta(z_s )}(z_s ) =\emptyset \text{ if } j\neq s.$

If some $Q_{200^{-2} \delta(z)}(z)\in \Gamma_1$ was discarded at the $j-$th stage, then $Q_{200^{-2} \delta(z)}(z) \cap Q_{200^{-2} \delta(z_j )}(z_j ) \neq \emptyset. $ Fixing  a point $z'$ in the intersection, by Lemma 3.4 (3) we have
    $$Q_{200^{-2} \delta(z)}(z) \subseteq  Q_{200^{-1} \delta(z' )}(z' ) \subseteq Q_{   \delta(z_j )}(z_j ). $$
 Therefore,
$$\Omega_r\subseteq \bigcup_{z\in\Omega_r } Q_{200^{-2} \delta(z)}(z)\subseteq \bigcup_s Q_{  \delta(z_s )}(z_s ).$$
This means that the sequence $\{z_s\}$ satisfies $(i)$.

 Now we show that the sequence $\{z_s\}$ satisfies $(ii)$. From Lemma 3.4 (2), clearly $Q_{200^2\,\delta(z_s)}(z_s)\subseteq \Omega_{r-c}.$
  For any $z\in \Omega_{r-c}$, let $$\Lambda_z=\{j:z\in Q_{200^2 \delta(z_j )}(z_j )\}\subseteq \mathbb{N}.$$
Using Lemma 3.4 $(1)$ and $(2)$, one sees  that $$Q_{200^2 \delta(z_j )}(z_j ) \subseteq Q_{200^3 \delta(z )}(z  );\quad\,\,\frac{\delta(z_j)}{3}<\delta(z )<3\delta(z_j)\quad\,\forall j\in\Lambda_z . $$
By the fact that $Q_{200^{-2} \delta(z_j )}(z_j )  \cap Q_{200^{-2} \delta(z_s )}(z_s ) =\emptyset,  \forall j,s\in\Lambda_z,j\neq s,
$
and $$\cup_{j\in\Lambda_z }Q_{200^{-3} \delta(z)}(z_j ) \subseteq \cup_{j\in\Lambda_z} Q_{200^{-2} \delta(z_j )}(z_j ) \subseteq Q_{200^3 \delta(z )}(z  ), $$ we have
 $|\Lambda_z|\leq \frac{Vol(Q_{200^3 \delta(z )}(z  ))}{Vol(Q_{200^{-3} \delta(z  )}(z   ))} =200^{6n+6},$  which establishes $(ii)$.
\end{proof}

Now we turn to the proof of {Proposition\,\,\ref{eq:original}$ {'}$}. Here we use the same notation as in Proposition 3.5.

\vskip2mm{\bf{Proof of {Proposition\,\,\ref{eq:original}$ {'}$ (2)}}}

    We begin with  a local result, i.e., an inequality which holds on $Q_{ \delta(z )}(z )$  with   $z=(a,0,0,\cdots,0).$  Obviously,   $L_{j,i}\neq 0 $ only if $i=1, j>1$ or $i>1, j=1$; and in these cases $L_{j,i}= \overline{a}\partial_j$ or $L_{j,i}= -\overline{a}\partial_i$, respectively.

    We consider the complex tangential derivative $ \partial_2$ first. For a point $w\in\mathbb{C}^n$, write $ w=(z_1,z_2,z')$ with $z_1,z_2\in\mathbb{C},z'\in\mathbb{C}^{n-2}$. For any $z_1,z'$ satisfying  $|z_1-a|< \delta(z ),|z'|<\sqrt{0.5 \delta(z )}$, we have $w\in Q_{ \delta(z )}(z )$ if  $|z_2| <\sqrt{0.5 \delta(z ) }.$

   Using   Lemma 3.2 one shows that,  if $|z_1-a|< \delta(z )$ and $|z'|<\sqrt{0.5 \delta(z )}$, then $$|\sqrt{0.5 \delta(z )}\partial_2 p(z_1,0,z') f(z_1,0,z')|\leq 2m\int_{|z_2|<\sqrt{0.5 \delta(z )}} |p(z_1,z_2,z')f(z_1,z_2,z')| \frac{dm(z_2)}{0.5\pi \delta(z_2 )  } .$$
   Therefore,
   \begin{eqnarray*}
  && |\sqrt{0.5 \delta(z)}\partial_2 p(a,0,0) f(a,0,0)|\\
  &\leq&   \int_{|z_1-a|< \delta(z),|z'|<\sqrt{0.5 \delta(z)}}|\sqrt{0.5 \delta(z)}\partial_2 p(z_1,0,z') f(z_1,0,z')| \frac{dm(z_1)}{\pi {\delta(z)}^2}\frac{dm(z')}{Vol\{|z'| <\sqrt{0.5 \delta(z)}\}}\\
  &\leq& 2^{ n }m \int_{w\in Q_{\delta(z)}(z)} |p(w)f(w)|\frac{dm(w)}{Vol(Q_{\delta(z)}(z))}.
   \end{eqnarray*}
Using  H\"{o}lder's  inequality, we have
   $$|\sqrt{0.5 \delta(z)}\partial_2 p(a,0,0) f(a,0,0)|^2\leq 2^{2n }m^2 \int_{w\in Q_{\delta(z)}(z)} |p(w)f(w)|^2 \frac{dm(w)}{Vol(Q_{\delta(z)}(z))}.$$
   The same argument is also valid for $\partial_j,\,1< j\leq n$. This implies that
   \begin{eqnarray*}  | \nabla_T p(z) f(z)|^2 (1-|z|)\leq \frac{     2^{2n+1}   m^2}{c} \int_{w\in Q_{\delta(z)}(z)} |p(w)f(w)|^2\frac{dm(w)}{Vol(Q_{ \delta(z)}(z))}.\end{eqnarray*}
The expression  $|\nabla_T p(z)|$ is called the tangential gradient of $p$ at $z$ (see e.g. \cite[Section 7.6]{Zhu}) with the definition
$$|\nabla_T p(z)| = max\{|\sum_{i=1}^n u_i\partial_i p(z)|:u\in\partial\mathbb{B}_n,u\bot z\}.$$
Using   rotation, the above inequality is valid for any $z\in\Omega_r$ with $r=\frac{1}{2}$.
  This means that for any $z\in\Omega_{1/2}, 1\leq i\neq j\leq n$, we have
  \begin{eqnarray*}  | L_{j,i} p(z) f(z)|^2 (1-|z|)\leq \frac{     2^{2n+1}   m^2}{c} \int_{w\in Q_{\delta(z)}(z)} |p(w)f(w)|^2\frac{dm(w)}{Vol(Q_{ \delta(z)}(z))}.\end{eqnarray*}

 Therefore, for   $  1\leq i\neq j\leq n$ one sees that
   \begin{eqnarray*}
   & &\int_{z\in Q_{ \delta(z_s)}(z_s)}   |L_{j,i} p(z) f(z)|^2  (1-|z|^2)^{2k+1} dm(z)\\
   &\leq &
 2   \int_{z\in Q_{ \delta(z_s)}(z_s)}   |L_{j,i} p(z) f(z)|^2 (1-|z|) (1-|z|^2)^{2k} dm(z)\\
   &\leq & {\textstyle\frac{    2^{2n+2}   m^2}{c}  (1+2c)^{2k} (1-|z_s|^2)^{2k} \int_{z\in Q_{ \delta(z_s)}(z_s)}\big{[} \int_{w\in Q_{ \delta(z)}(z)} |p(w)f(w)|^2 \frac{dm(w)}{Vol(  Q_{ \delta(z)}(z))}\big{]} dm(z)}\\
     &\leq & {\textstyle\frac{    2^{2n+2}    m^2(1+2c)^{2k}}{c}   (1-|z_s|^2)^{2k} \int_{z\in Q_{ \delta(z_s)}(z_s)} \big{[} \int_{w\in Q_{200 \delta(z_s)}(z_s)} |p(w)f(w)|^2 \frac{  dm(w)}{Vol(Q_{ \delta(z)}(z))} \big{]}   dm(z)}\\
    &\leq & \frac{   3^{n+1} 2^{2n+2}    m^2(1+2c)^{2k}}{c(1-3 c)^{2k}}   \int_{w\in Q_{200 \delta(z_s)}(z_s)} |p(w)f(w)|^2 (1-|w|^2)^{2k}  dm(w).
   \end{eqnarray*}
   By   Covering Lemma \ref{covering}, we have
     \begin{eqnarray*}
      &&\int_{z\in \Omega_{1/2}}   |L_{j,i} p(z) f(z)|^2  (1-|z|^2)^{2k+1} dm(z)\\
      & \leq &{\textstyle\sum_s \int_{z\in Q_{ \delta(z_s )}}   |L_{i,j} p(z) f(z)|^2  (1-|z|^2)^{2k+1} dm(z)}\\
    &\leq & {\textstyle\sum_s \frac{   3^{n+1} 2^{2n+2}    m^2(1+2c)^{2k}}{c(1-3 c)^{2k}} \int_{z\in Q_{200 \delta(z_s )}(z_s )} |p(z)f(z)|^2 (1-|z|^2)^{2k}  dm(z)}\\
    &\leq &{\textstyle \frac{   3^{n+1} 2^{2n+2}    m^2(1+2c)^{2k}}{c(1-3 c)^{2k}}    N(n)\int_{\mathbb{B}_n} |p(z)f(z)|^2 (1-|z|^2)^{2k}  dm(z)}
    \\&\leq & {\textstyle(24^{ n+1}      m^2 N(n)/c)^{k+1} \int_{\mathbb{B}_n} |p(z)f(z)|^2 (1-|z|^2)^{2k}  dm(z),}
     \end{eqnarray*}
     as desired.

To prove   inequality $(1)$, the following lemma is needed.
\begin{lem}
    For any smooth function $f$ on the complex plane $\mathbb{C}$,
             $$R^l f =\sum_{j=1}^l a^{(l)}_j z^j \partial^j f $$
   with $|a^{(l)}_j|<(j+1)^l.$
\end{lem}
\begin{proof}
We prove the lemma by induction on $l$. Clearly it holds in the case $l=1$. Suppose the inequality  holds for the coefficients for $l=s$. For $l=s+1$, we have
\begin{eqnarray*}
R^{s+1} f&= &z\partial (\sum_{j=1}^s a^{(s)}_j z^j \partial^j f)\\
   &=&z  (\sum_{j=1}^s a^{(s)}_j j z^{j-1} \partial^j f+\sum_{j=1}^s a^{(s)}_j z^j \partial^{j+1} f) \\
   &=& \sum_{j=1}^{s+1} (j a^{(s)}_j +a^{(s)}_{j-1})z^j \partial^{j } f,
\end{eqnarray*}
where we are assuming that   $  a^{(s)}_{0} =a^{(s)}_{s+1}=0.$ By  the induction hypothesis $  |a^{(s)}_j|<(j+1)^s,$ one sees that
$$|a_j^{(s+1)}|= |j a^{(s)}_j +a^{(s)}_{j-1}| <(j+1)^{s+1}, $$
which completes  the proof of the lemma.
\end{proof}

     Now we return to  prove   inequality $(1)$.

\vskip2mm{\bf{ Proof of {Proposition\,\,\ref{eq:original}$ {'}$ (1)}}}

  We first  reduce the question to the case of dimension one. Indeed, define the slice function $g_\xi(z)=g(\xi z)$ for $g\in C(\overline{\mathbb{B}_n})$ and $\xi\in\partial \mathbb{B}_n,\,z\in\mathbb{D}$. Using Propositions 1.4.3 and Proposition 1.4.7(1) in \cite{R1}, we have that for $g\in C(\overline{\mathbb{B}_n})$
   \begin{eqnarray}\label{slice}
   \int_{\mathbb{B}_n} g dm &=&  {2n}{Vol(  \mathbb{B}_n)} \int_{r\in [0,1]} r^{2n-1} dr \int_{\xi\in\partial \mathbb{B}_n } g(r\xi) d\sigma(\xi)\\
           &=&  {2n}{Vol(  \mathbb{B}_n)} \int_{r\in [0,1]}  r^{2n-1} dr \int_{\xi\in\partial \mathbb{B}_n } d\sigma(\xi)\int_{\theta\in (-\pi,\pi]} g(re^{i\theta}\xi)\frac{d\theta}{2\pi}\nonumber\\
           &=& \frac{1}{ 2\pi  } \int_{\xi\in\partial \mathbb{B}_n } dm(\xi) \int_{z\in\mathbb{D}} g_{\xi}( z) |z^{n-1}|^2 dm(z),\nonumber
   \end{eqnarray}
   where $d\sigma(\xi) =\frac{dm(\xi)}{ {Vol(\partial  \mathbb{B}_n)}} =\frac{dm(\xi)}{{2n}{Vol(  \mathbb{B}_n)}}$ is the normalized Lebesgue measure on $\partial \mathbb{B}_n $.

    Noticing that $R_z ( p_\xi (z))= (Rp)_\xi(z)$, where $R_z$ is the radial derivative in the one variable $z$,  by formula ($\ref{slice}$)   we have
     \begin{eqnarray*}
     &&\int_{\mathbb{B}_n}  |(R^k p)(z)\,f(z)|^2   (1-|z|^2)^{2k} dm(z)\\ &=&\frac{1}{2\pi}\int_{\xi\in \partial \mathbb{B}_n } \Big{[}\int_{z\in \mathbb{D}} |R_z^k(  p_\xi(z))\,f_{\xi}(z) z^{n-1}|^2   (1-|z|^2)^{2k} dm(z)\Big{]} dm(\xi);\\
   &&\int_{\mathbb{B}_n}  |  p (z)\,f(z)|^2    dm(z)
   \\&=&\frac{1}{2\pi} \int_{\xi\in \partial \mathbb{B}_n }\Big{[}\int_{z\in \mathbb{D}} |  p_{\xi} (z)\,f_{\xi}(z) z^{n-1}|^2     dm(z)\Big{]} dm(\xi).
     \end{eqnarray*}
    So, it suffices to show the inequality involving one variable functions.

   Now we use the Covering Lemma to show the inequality on  $\mathbb{C}$ for $\partial^j,\,1\leq j\leq \min(m,l) $ .  In this case, the covering domains in Proposition 3.5 degenerate to   disks with radii $ \delta(z).$  The same argument as in the proof of Proposition $2. 2'\,\,\, (2)$ shows that  for $1\leq j\leq min(m,l), $ one has that for $\     z\in\mathbb{D}_\frac{1}{2}=\{w\in\mathbb{D}:|w|>\frac{1}{2}\}$
   \begin{eqnarray*} | \partial^j\,p(z) f(z)|^2 \delta^{2j}(z )\leq      [\frac{(j+2)m!}{2(m-j)! }]^2 \int_{w\in Q_{ \delta(z )}(z )} |p(w)f(w)|^2\frac{dm(w)}{Vol(  Q_{ \delta(z )}(z ))}.  \end{eqnarray*}
      This implies that if $1\leq j\leq min(m,l)\leq k, $ then we have
 \begin{eqnarray*}
   & &\int_{z\in Q_{ \delta(z_s )(z_s )}} | \partial^j\,p(z) f(z)|^2  (1-|z|^2)^{2k} dm(z)\\
        &\leq &   \frac{3^{2}(m+1)!^2 2^{2j} (1+2c)^{2k-2j}}{c^{2j}(1-3 c)^{2k-2j}} \int_{w\in Q_{200 \delta(z_s )}(z_s )} |p(w)f(w)|^2 (1-|w|^2)^{2k-2j}  dm(w),
   \end{eqnarray*}
   and hence
       \begin{eqnarray*}
   & &\int_{z\in\mathbb{D}_\frac{1}{2}} |(\partial^j p)(z)\,f(z)|^2   (1-|z|^2)^{2k} dm(z)\\
        &\leq &  \frac{3^{2}(m+1)!^2 2^{2j} (1+2c)^{2k-2j}}{c^{2j}(1-3 c)^{2k-2j}}  N(n) \int_{ \mathbb{D}   } |p(z)f(z)|^2  (1-|z|^2)^{2k-2j}  dm(z) \\&\leq &  (12^{2} (m+1)! ^2  N(n)/c^2 )^{k+1}\int_ \mathbb{D}  |p(z)f(z)|^2  (1-|z|^2)^{2k-2l}  dm(z)         .
          \end{eqnarray*}

    Using Lemma 3.6 we show  that for the polynomial $p$ with $m=\deg(p)$
     $$|R^l p|  =|\sum_{j=1}^{\min\{l,m\}} a^{(l)}_j z^j \partial^j p|\leq  (m+1)^l  \sum_{j=1}^{\min\{l,m\}} | \partial^j p |.$$
     Therefore,  one has \begin{eqnarray*}
   & &\int_{\{z\in\mathbb{D}:|z|>\frac{1}{2}\} } |(R^l p)(z)\,f(z)|^2   (1-|z|^2)^{2k} dm(z)\\
        &\leq & (12^{2} m^2(m+1)(m+1)! ^2  N(n)/c^2)^{k+1} \int_{\mathbb{D} } |p(z)f(z)|^2  (1-|z|^2)^{2k-2l}  dm(z), \end{eqnarray*}
     completing  the proof of $(1).$ \qquad\qquad\qquad\qquad\qquad\qquad\qquad\qquad\qquad\qquad\qquad\qquad$\Box$\vskip3mm

It  remains to prove $(3)$. One can prove it   using the above methods  or it can be shown   directly from Proposition $2.2'\,\,\, (1)(2)$ as follows.

   \vskip2mm{\bf{ Proof of {Proposition\,\,\ref{eq:original}$ {'}$ (3)}}}

     By  equation $(\ref{eq:derivative} )$,  we have $|z|^2\partial_j p=\overline{z_j} Rp+\sum_{i=1,i\neq j}^n z_i L_{j,i}p,$ which implies that
     $$|\partial_j p|\leq 4|R p|+\sum_{i\neq j} 4|L_{j,i}p|$$
     for $|z|>1/2.$  Combing this inequality  with Proposition $2.2'\,\,\,(1),(2)$  shows the desired result.\quad\quad\quad\quad\quad\quad\quad\quad\quad\quad\quad\quad\quad\quad\quad\quad\quad\quad\quad\quad\quad\quad\quad\quad\quad\quad\quad\quad\quad\quad\quad\quad\quad\quad$\Box$

\section{Further discuss}
\subsection{The weighted Bergman space $L_a^2(\mu_p)$}
For $p\in\mathbb{C}[z_1,\cdots,z_n]$, let $L^2(\mu_p)$ be the Hilbert space consisting of  functions having the property that $\int_{\mathbb{B}_n}|f|^2 d\mu_p<\infty$,   where $\mu_p$ is the measure on $\mathbb{B}_n$ defined by $d\mu_p=|p|^2dm$, and let $L^2_a(\mu_p)$ be the weighted Bergman space consisting of the analytic functions in $L^2(\mu_p)$. Little is known about this natural analytic function space.
In what follows, we show some elementary properties of  $ L^2_a(\mu_p)$ using the methods and results in Section 3.
\begin{lem}
For a  polynomial $p\in\mathbb{C}[z_1,\cdots,z_n]$ with $m=\deg(p)$, we have  for any $f\in L^2_a(\mu_p)$ that
$$ \int_{\mathbb{B}_n} | f_r|^2 |p|^2 dm\leq 2^{2(m+n-1)} \int_{\mathbb{B}_n} | f |^2 |p|^2 dm, \,\,\,\text{ if }\frac{1}{2}<r<1,$$
where $f_r(z)=f(rz)$ for $z\in\mathbb{B}_n$.
\begin{proof}
 Firstly we show the inequality in the case of
one dimension as follows.

For each polynomial $g$ with $m=\deg(g)$, suppose $$g(z)=z^u(z-a_1)\cdots(z-a_v)(z-b_1)\cdots(z-b_s),$$ where $u+v+s=m, |a_i|\geq 1,|b_i|<1,b_i\neq 0$.
Let $$\widetilde{g} (z)=(z-a_1)\cdots(z-a_v)(1-\overline {b_1} z) \cdots(1-\overline{b_s}z).$$ By Lemma 2.1 in \cite{Ge}, one sees that $\frac{\widetilde{g}(z)}{\widetilde{g}(rz)}\leq 2^m$ for $\frac{1}{2}<r<1,|z|\leq 1$. This implies that for $h\in A(\mathbb{D})$ and $\frac{1}{2}<r<1$, we have
\begin{eqnarray*}
\int_{\mathbb{T}} |g(e^{i\theta}) h(r e^{i\theta})|^2 \frac{dm(\theta)}{2\pi}&=&\int_{\mathbb{T}} |\widetilde{g}(e^{i\theta}) h(r e^{i\theta})|^2 \frac{dm(\theta)}{2\pi}
    \leq  2^{2m}\int_{\mathbb{T}} |\widetilde{g}(re^{i\theta}) h(r e^{i\theta})|^2 \frac{dm(\theta)}{2\pi}\\
   &\leq&  2^{2m}\int_{\mathbb{T}} |\widetilde{g}( e^{i\theta}) h(  e^{i\theta})|^2 \frac{dm(\theta)}{2\pi}=2^{2m} \int_{\mathbb{T}} |g(e^{i\theta}) h( e^{i\theta})|^2\frac{dm(\theta)}{2\pi}.
\end{eqnarray*}
Therefore, for $f\in L^2_a(\mu_p)$ and $\frac{1}{2}<r<1,$ one has
 \begin{eqnarray*}
 &&\int_{\mathbb{D}} |p(z) f(r\,z)|^2 \frac{d m(z)}{\pi}
 =\int_{0<r'<1 }\big{[}\int_{\mathbb{T}}|p(r'\, e^{i \theta}) f(r\,r'\, e^{i \theta})|^2\frac{dm(\theta)}{2\pi}\big{]} {2rdr}\\
 &\leq& 2^{2m}\int_{0<r'<1 }\big{[}\int_{\mathbb{T}}|p(r'\, e^{i \theta}) f(  r'\, e^{i \theta})|^2\frac{dm(\theta)}{2\pi}\big{]} {2rdr}=2^{2m}\int_{\mathbb{D}} |p(z) f( z)|^2 \frac{d m(z)}{\pi},
\end{eqnarray*}
which establishes  the inequality   in the case of one dimension. \\Now we prove the general case by a slice argument as in  formula $(\ref{slice})$.
Indeed, we have that
 \begin{eqnarray*}
&& \int_{\mathbb{B}_n} | f_r|^2 |p|^2 dm=\frac{1}{2\pi}\int_{\xi\in\partial \mathbb{B}_n}dm(\xi)\int_{z\in\mathbb{D}}|f(\xi r z)|^2|z^{n-1}p(\xi z)|^2dm(z)\\
 &\leq&\frac{2^{2(m+n-1)}}{2\pi}\int_{\xi\in\partial \mathbb{B}_n}dm(\xi) \int_{z\in\mathbb{D}}|f(\xi   z)|^2| z^{n-1}p(\xi z)|^2dm(z)
= 2^{2(m+n-1)}\int_{\mathbb{B}_n} | f |^2 |p|^2 dm,
 \end{eqnarray*}
 which completes the proof.
\end{proof}
\end{lem}
\begin{lem}
The weighted Bergman space $L^2_a(\mu_p)$  is complete.
\end{lem}
\begin{proof}
It suffices to show that $L^2_a(\mu_p)$  is a closed subspace of $L^2(\mu_p)$. That is, if a sequence $\{f_n\}$ in $L^2_a(\mu_p)$ converges to $f$  in the norm of $L^2(\mu_p)$,
 then $f$ is   equal a.e. to an analytic function on the unit ball. Choose a multi-index $\alpha $ such that $|\alpha|=\deg(p)$ and $\partial^\alpha p$ is a nonzero constant. Using the above lemma and Proposition $2.2(3)$, we have for any $\frac{1}{2}<r<1$, that
\begin{eqnarray*}
&&|\partial^\alpha p|^2 \int_{\mathbb{B}_n} |f_n(r\, z)-f_l(r\,z)|^2 (1-|z|^2)^{2|\alpha|}dm(z) \\ &\leq&
 c_{2|\alpha|}\,\,  {\textstyle\prod_{m=1}^{|\alpha|} C(n,m)^{1+|\alpha|-m}}\int_{\mathbb{B}_n} |f_n(r\, z)-f_l(r\,z)|^2|p(z)|^2 dm(z)\\&\leq& 2^{2(|\alpha|+n-1)} c_{2|\alpha|} \,\, {\textstyle\prod_{m=1}^{|\alpha|} C(n,m)^{1+|\alpha|-m}}\int_{\mathbb{B}_n} |f_n( z)-f_l(z)|^2|p(z)|^2 dm(z) \rightarrow 0
\end{eqnarray*}
   as  $n,l\to \infty$, where $ C(n,m)$ is the constant appearing in Proposition $2.2$. This implies that the sequence $ f_n  $ is pointwise convergent to an analytic function $g$. Noticing that $f_n$ is also pointwise convergent to $f$ outside the zero measure set  $Z(p)\cap \mathbb{B}_n$,   we have $f=g\,a.e.$, which completes the proof.
\end{proof}
\begin{lem}
The polynomial ring $ \mathbb{C}[z_1,\cdots,z_n]$ is dense in $L^2_a(\mu_p)$.
\end{lem}
\begin{proof}
Let $M$ be the closure of $ \mathbb{C}[z_1,\cdots,z_n]$ in $L^2_a(\mu_p)$. Obviously, for each $g\in A(\mathbb{B}_n)$, we have $g\in M$. For any $f\in L^2_a(\mu_p)$,  set $f_n(z)=f((1-\frac{1}{n})z)$.
By Lemma 4.1 the sequence  $ {f_n } $ is   uniformly bounded  in $L^2_a(\mu_p)$. So, there exists a subsequence $f_{n_k}$ which is weakly convergent to
some function $g\in L^2_a(\mu_p)$. Clearly $g\in M$. Moreover, for each $z\in \mathbb{B}_n$, by the proof for the above lemma, the point evaluation at $z$ is a
bounded functional in the Hilbert space  $L^2_a(\mu_p)$. This implies that $f_{n_k}(z)\to g(z)$ for each $z\in \mathbb{B}_n$. Thus, $g=f$ and hence $f\in M$.
This means that the closure $M=L^2_a(\mu_p)$.
\end{proof}
\vskip2mm
We summarize the results in this subsection in the following.
\begin{thm}
Let  $p\in \mathbb{C}[z_1,\cdots,z_n]$. Set $d\mu_p=|p|^2 dm$ and $$L_a^2(\mu_p)=\{f\in L^2(\mu_p),\text{ f holomorphic on  }\mathbb{B}_n\}.$$ Then
$L_a^2(\mu_p)$ is a reproducing kernel Hilbert space on $\mathbb{B}_n$, which defines a  $p$-essentially normal Hilbert module whose essential spectrum
equals   $\partial \mathbb{B}_n$. Moreover, $ \mathbb{C}[z_1,\cdots,z_n]$ is dense in $L^2_a(\mu_p)$. And $L_a^2(\mu_p)\subset L_{a,t}^2(\mathbb{B}_n)$ for $t\geq 2 \deg(p).$
\end{thm}
\begin{proof} Consider the operator $ \mathcal {I}:L^2_a(\mu_p) \to L^2_a(\mathbb{B}_n)$  defined by
$$\mathcal {I} (p)=pf. $$
This natural embedding map $\mathcal {I}$ is an isometrical  module isomorphism from $L^2_a(\mu_p)$ to the image $   \mathrm{ran}(\mathcal{I})$. Clearly the submodule $[p]\subseteq  \mathrm{ran}(\mathcal{I}).$ Furthermore, by Lemma 4.3, each of $[p]$ and  $\mathrm{ran}(\mathcal{I}) $ is the closure of the ideal   $p\mathbb{C}[z_1,\cdots,z_n]$. This means that $[p]= \mathrm{ran}(\mathcal{I}).$ Hence, by Theorem 2.5 one sees that $L^2_a(\mu_p)$ is essentially normal, which is  a result  analogous to the basic result for the  Bergman space.
\end{proof}

Moreover, we also have  obtained a somewhat surprising  result in function theory since   $[p]= \mathrm{ran}(\mathcal{I})$.
 \begin{cor}  For any analytic function $f\in L^2_a(\mathbb{B}_n)$, one has that $f\in [p]$ if and only if $f=ph$ for  some analytic function $h$ on $\mathbb{B}_n$.\end{cor}

   \subsection{Quotient Modules}
       Let $p\in\mathbb{C}[z_1,\cdots,z_n],\mathcal{M}_p=[p]\subseteq L_a^2(\mathbb{B}_n)$ be the cyclic submodule generated by $p$, $\mathscr{Q}_p$
       be the quotient module defined by the short exact sequence
       $$ 0\longrightarrow \mathcal{M}_p \longrightarrow L_a^2(\mathbb{B}_n)\longrightarrow \mathscr{Q}_p\longrightarrow 0,$$
       and $Q_f$ be the compression of $M_f$ on $L_a^2(\mathbb{B}_n)$ to $\mathscr{Q}_p$ for $f\in H^\infty(\mathbb{B}_n).$ Then the map $f\to Q_f$ for $f\in\mathbb{C}[z_1,\cdots,z_n]$ defines the module  action of $\mathbb{C}[z_1,\cdots,z_n]$ on $\mathscr{Q}_p$.

        Let $\mathscr{T}(\mathscr{Q}_p)$ be the $C^*$-subalgebra of $\mathscr{L}(\mathscr{Q}_p)$ generated by $\{Q_f:f\in\mathbb{C}[z_1,\cdots,z_n]\}$ and $ \mathcal {K}(\mathscr{Q}_p)$ be the ideal of compact operators on $\mathscr{Q}_p$.        From Theorem 2.5 and Lemma 2.1 in \cite{GW} or the related result in \cite{Ar2,Dou1,Guo,GWk2},  it follows that all the operators $Q_f$ are essentially normal, or $[Q_f,Q_g^\ast]\in \mathcal {K}(\mathscr{Q}_p)$ for $f,g\in
          \mathbb{C}[z_1,\cdots,z_n]$, and hence $\mathscr{T}(\mathscr{Q}_p)/ \mathcal {K}(\mathscr{Q}_p)$ is a commutative $C^\ast$-algebra. This means
          that it's isometrically isomorphic to $C(X_p)$ for some compact metrizable  space $X_p$. Using the image of the $n$-tuple $(Q_{z_1},\cdots,Q_{z_n})$ in $\mathscr{T}(\mathscr{Q}_p)/ \mathcal {K}(\mathscr{Q}_p) $, we can identify $X_p$ as  a subset of $\mathbb{C}^n$. Moreover, since $\sum_{i=1}^n Q_{z_i}^\ast Q_{z_i}\leq I$, one sees
          that $X_p\subseteq \, clos\,\mathbb{B}_n$. In fact, we have the following partial characterization of $X_p$.

          \begin{prop}
          For $p\in\mathbb{C}[z_1,\cdots,z_n]$, we have
          $$clos \{Z(p)\cap \mathbb{B}_n\}\cap \partial \mathbb{B}_n \subseteq X_p\subseteq Z(p)\cap \partial \mathbb{B}_n .$$
          \end{prop}

          Note that a point $z_0$ is in $Z(p)\cap \partial \mathbb{B}_n$  and not in $ clos \{Z(p)\cap \mathbb{B}_n\}$ only when
          the component of $Z(p)$ containing $z_0$ is "tangent" to $\mathbb{B}_n$ in some sense.

          \begin{proof}
          For $f\in\mathbb{C}[z_1,\cdots,z_n]$, we can write
          $$M_f=S_f\oplus Q_f +K,$$
          where $K\in \mathcal{K}(L_a^2(\mathbb{B}_n))$. Since the $C^*-$ algebra generated by $\{M_f:f\in\mathbb{C}[z_1,\cdots,z_n]\}$ contains
          $\mathcal{K}(L_a^2(\mathbb{B}_n))$ and $\mathscr{T}(L_a^2(\mathbb{B}_n))/ \mathcal {K}(L_a^2(\mathbb{B}_n))\cong C(\partial \mathbb{B}_n)$, we have a $*-$ homomorphism from $C(\partial \mathbb{B}_n) $ to $C(X_p)$.  It follows that $$X_p\subseteq\sigma_e\{M_{z_1},\cdots,M_{z_n}\} = \partial \mathbb{B}_n,$$
          where $\sigma_e$ denotes the joint essential spectrum.

          If $z_0=(z^0_1,\cdots,z^0_n)\in\partial\mathbb{B}_n$ such that $p(z_0)\neq 0$, then the ideal in $\mathbb{C}[z_1,\cdots,z_n]$ generated by $\{z_1-z^{0}_1,\cdots, z_n-z_n^0,p\}$ equals   $\mathbb{C}[z_1,\cdots,z_n]$. Therefore, there exist polynomials $\{q_i\}_{i=1}^{n+1}$ such
          that $$\sum_{i=1}^n q_i(z)(z_i-z_i^0)+ q_{n+1}(z)p(z)\equiv 1.$$

        This implies that $\sum_{i=1}^n Q_{q_i}Q_{z_i-z^0_i}=I_{\mathscr{Q}_p}$, or $z_0$ is not in the joint essential spectrum of the $n$-tuple $\{Q_{z_1},\cdots,Q_{z_n}\}$
        and $z_0\notin X_p$.

        Suppose $w_0=(w_1^0,\cdots,w_n^0)\in\partial \mathbb{B}_n$ such that there exists $\{w_k\}_{k=1}^\infty \subseteq  Z(p)\cap \mathbb{B}_n$ and $w_k\rightarrow w_0$. Let $\{\xi_k\}$
        be   unit vectors in $L_a^2(\mathbb{B}_n)$ such that $M^\ast_f \xi_k=\overline{f(w_k)}\xi_k$ for $f\in\mathbb{C}[z_1,\cdots,z_n]$ and $k\in\mathbb{N}$. It's well known that
        $\xi_k$ is weakly convergent  to $0$ since $\overline{w_0}$ is not a joint eigenvalue of the $n-$tuple $(M_{z_1}^\ast,\cdots,M_{z_n}^\ast)$.  Since
        $$\langle \xi_k, pf\rangle_{L_a^2(\mathbb{B}_n)}=\langle M^\ast_p\xi_k,  f\rangle_{L_a^2(\mathbb{B}_n)}=\overline{p(\xi_k)}\langle \xi_k, f\rangle_{L_a^2(\mathbb{B}_n)}=0, \forall f\in L_a^2(\mathbb{B}_n)$$
        we have $\xi_k \perp [p]$ and hence $\{\xi_k\}\subseteq \mathscr{Q}_p$. Moreover, $Q^\ast_f\xi_k=M^\ast_f\xi_k=\overline{f(\xi_k)}\xi_k$ for $k\in\mathbb{N},f\in \mathbb{C}[z_1,\cdots,z_n]$.

         Now we claim that such $w_0=(w_1^0,\cdots,w_n^0) \in X_p$. Otherwise,  the $n$-tuple of operators $ (Q_{z_1-w^0_1},\cdots,Q_{z_n-w^0_n}) $ is Fredholm and hence
         the range of $H=\sum_{i=1}^n Q_{z_i-w^0_i}Q_{z_i-w^0_i}^\ast$ has finite codimension  in $\mathscr{Q}_p$. Thus  there exists a finite rank projection $E$ and $\varepsilon>0$
         such that $H+E>\varepsilon I_{\mathscr{Q}_p}$. However, a direct computation shows that
         $$\langle(H+E)\xi_k,\,\xi_k\rangle=\sum_{i=1}^n |w_i^k-w_i^0|^2+ \langle   E \,\xi_k,\,\xi_k\rangle\rightarrow 0,$$
         since $E$ is a finite rank operator and $\xi_k\to 0$ weakly.  This leads to a contradiction. Therefore, we have $w_0  \in X_p$, completing the proof of the proposition.\end{proof}

         In many cases, the two sets are equal and thus $X_p$ is characterized completely.

         Recall that $f\in\mathbb{C}[z_1,\cdots,z_n]$ is said to be quasi-homogeneous if there exists $k_1,\cdots,k_n\in\mathbb{N}$ and a homogeneous polynomial $g\in\mathbb{C}[z_1,\cdots,z_n]$ such that $f(z_1,\cdots,z_n)=g(z_1^{k_1},\cdots,z_n^{k_n})$ for $(z_1,\cdots,z_n)\in\mathbb{C}^n$.

         \begin{cor}
         For a quasi-homogeneous polynomial $p\in\mathbb{C}[z_1,\cdots,z_n]$, we have $X_p=Z(p)\cap \partial \mathbb{B}_n$.
         \end{cor}
         \begin{proof}
         Suppose that $p(z_1,\cdots,z_n)=g(z_1^{k_1},\cdots,z_n^{k_n})$ for some homogeneous polynomial $g$ and $k_1,\cdots,k_n\in\mathbb{N}$. For any
         $z_0=\{z^0_1,\cdots,z^0_n\}\in Z(p)\cap \partial\mathbb{B}_n$,  we have $g((z^0_1)^{k_1},\cdots,(z^0_n)^{k_n})=0$, which implies that $g(r(z^0_1)^{k_1},\cdots,r(z^0_n)^{k_n})=0$ for $0<r<1$ since $g$ is homogeneous. This means that $p(z^r)=0$ for $z^r=(r^\frac{1}{k_1}z^0_1,\cdots,r^\frac{1}{k_n}z^0_n)$, and $z^r\to z^0$, which completes the proof.
         \end{proof}

         Since $p\in\mathbb{C}[z_1,\cdots,z_n]$ defines the extension of $\mathcal{K}(\mathscr{Q}_p)$ by $C(X_p)$, we have $[p]\in K_1(X_p)$, the odd K-homology group of the compact metrizable space $X_p$ \cite{BDF1}. A basic question is to determine which element one has. In \cite{Dou} it was conjectured   that $[p]$ is the fundamental class of $X_p$ determined by the almost complex structure of $X_p\subset\partial\mathbb{B}_n$. (In \cite{Dou} the multiplicity of $p$ was not taken into account. For example,  one sees that $X_{p^2}=X_p$ for $p\in\mathbb{C}[z_1,\cdots,z_n]$ but $[p^2]=2[p]\in K_1(X_p).$ ) In \cite{GW}, the element $[p]$ is calculated for the case $p$ is homogeneous and  $n=2$ and in this case one can show that  $[p]$ equals   the fundamental class. In this case $X_p$ consists of the union of a finite number of circles. Hence $[p]\in K_1(X_p)$ is determined by the index of an appropriate  operator for each circle with the property that the fundamental class is determined by the "winding number" of the polynomial on these circles. The basic technique in \cite{GW} is to first factor $p(z_1,z_2)$ and reduce the calculation to that of a single factor.

          The proposition raises a number of questions which we now discuss briefly.

          First, is it always the case that $X_p= clos (Z(p)\cap \mathbb{B}_n) \cap \partial \mathbb{B}_n$?  This question is closely related to the question of whether $p\in\mathbb{C}[z_1,\cdots,z_n]$ with $Z(p)\cap\mathbb{B}_n=\phi$   is cyclic which was answered in the affirmative in \cite{CG}. What would be needed to solve the question here would be a technique which allows one to handle the case in which some points in $Z(p)\cap \partial \mathbb{B}_n$ are "tangent" to $\partial\mathbb{B}_n$ but others   are not.

          Moreover, we observe that the above proposition carries over to    many  more general submodules of $L^2_a(\mathbb{B}_n)$. For example, for $\phi_1,\cdots,\phi_k\in A(\mathbb{B}_n),$ the ball algebra of functions continuous on $clos(\mathbb{B}_n)$  and holomorphic on $\mathbb{B}_n$, one can see that
              $$clos (Z(\phi_1,\cdots,\phi_k)\cap \mathbb{B}_n) \cap \partial \mathbb{B}_n \subseteq X_{[\phi_1,\cdots,\phi_k]}\subseteq  Z(p)  \cap \partial \mathbb{B}_n,$$
              where $[\phi_1,\cdots,\phi_k]$ denotes the submodule of $L^2_a(\mathbb{B}_n)$ generated by $\phi_1,\cdots,\phi_k$ and $Z(\phi_1,\cdots,\phi_k)$ is the subset of $clos\mathbb{B}_n$ of common zeros of $\phi_1,\cdots,\phi_k $.  Similarly,  the question  whether the maximal ideal space $X_{[\phi_1,\cdots,\phi_k]}=clos (Z(\phi_1,\cdots,\phi_k)\cap \mathbb{B}_n) \cap \partial \mathbb{B}_n   $  is   related to the question of whether  $Z(\phi_1,\cdots,\phi_k)\cap\mathbb{B}_n=\phi$ implies $[\phi_1,\cdots,\phi_k]=L^2_a(\mathbb{B}_n)$, which is still open for the dimension $n>2$.
              The above argument also extends to other reproducing kernel Hilbert modules such as the Hardy and Drury-Arveson spaces.

              Second, the conjecture of Arveson concerns the closure of homogeneous polynomial ideals in the Drury-Arveson space. One can show in the case of homogeneous ideals, essential normality of the closure in the Hardy, Bergman and Drury-Arveson spaces are all equivalent. But this argument    doesn't work for the case of ideals generated by arbitrary   $p\in\mathbb{C}[z_1,\cdots,z_n]$. It seems likely that the argument in this paper can be generalized to obtain the same result for the Hardy   and the Drury-Arveson spaces. However, while we believe that both results hold, perhaps techniques  from   \cite{CSW,CFO}  may be needed  to complete the  proofs.

              Thirdly, in \cite{Dou} the first author offered a refined conjecture for the closure of homogeneous polynomial ideals in the Drury-Arveson space. Arveson conjectured that the commutators and cross-commutators for the operators $Q_f$ in $\mathscr{Q}_p$ were in the Schatten $p-$class for $p>n$ which we have established  in this paper for the case of principal polynomial ideals. However, in \cite{Dou}, it was conjectured that this result on the commutators  actually holds for $p>\dim Z(p)$. Although it is not clear if one can modify the proof herein to obtain this result, the question makes sense.

                Finally, it is natural to ask if the result in this paper extends to all ideals in $\mathbb{C}[z_1,\cdots, z_n]$ or even to all ideals in  $A(\mathbb{B}_n)$. One approach to this problem was discussed in \cite{DS}. A question, seemingly beyond current techniques, is whether a submodule of $L^2_a(\mathbb{B}_n)$ is essentially normal if and only if it is finitely generated. However, for the case $n=1$, the equivalence  holds with one direction following  from the Berger-Shaw Theorem \cite{BS}  and the other  from the result in \cite{ARS}.

                \vskip3mm {\bf{Acknowledgments}}\vskip3mm This research was  supported by
 NSFC (10731020,10801028),   the Department of Mathematics  at Texas A\&M University and Laboratory of Mathematics for Nonlinear Science at Fudan University.  The second author would like to thank the   host department for
its kind hospitality while the second author  visited the first author.

 \bibliographystyle{plain}

\end{document}